\documentclass[12pt,reqno, oneside]{amsart}
\usepackage{pdfsync}  
\usepackage{geometry}
\usepackage{hyperref}
\usepackage{diagbox}
\usepackage{subcaption}
\usepackage{mathrsfs} 
\usepackage{amssymb,amsmath,amsthm}

\newtheorem{thm}{Theorem}
\newtheorem{cor}[thm]{Corollary}
\newtheorem{lem}[thm]{Lemma}
\newtheorem{prop}[thm]{Proposition}

\let\oldtocsection=\tocsection
\let\oldtocsubsection=\tocsubsection
\let\oldtocsubsubsection=\tocsubsubsection
 
\renewcommand{\tocsection}[2]{\hspace{0em}\oldtocsection{#1}{#2}}
\renewcommand{\tocsubsection}[2]{\hspace{2em}\oldtocsubsection{#1}{#2}}
\renewcommand{\tocsubsubsection}[2]{\hspace{4em}\oldtocsubsubsection{#1}{#2}}

\newcommand{\dpow}[2]{\frac{#1^{#2}}{#2!}}
\newcommand{\dpowq}[2]{\frac{#1^{#2}}{#2!_q}}
\newcommand{\dpowg}[2]{\frac{#1^{#2}}{#2!\,2^{\binom n2}}}
\newcommand{\dpowgy}[2]{\frac{#1^{#2}}{#2!\,(1+y)^{\binom n2} }}
\newcommand{\qbinom}[2]{\binom{#1}{#2}_{\!\!q}}
\newcommand{\tqbinom}[2]{\binom{#1}{#2}_{\!q}} 
\newcommand{\ubinom}[2]{\binom{#1}{#2}_{\!\!u}}  
\newcommand{\tubinom}[2]{\binom{#1}{#2}_{\!u}}  
\newcommand{\dpowbq}[2]{\frac{#1^{#2}}{#2!_q(1+y)^{\binom n2}  }}
\newcommand{\floor}[1]{\left\lfloor#1\right\rfloor}

\newcommand\sumz[1]{\sum_{#1=0}^\infty}
\newcommand{\inverse}{^{-1}}

\newcommand{\B}{\mathscr{B}}

\DeclareMathOperator{\source}{s}
\DeclareMathOperator{\e}{e}
\DeclareMathOperator{\des}{des}
\DeclareMathOperator{\ssc}{ssc}

\title{Counting acyclic and strong digraphs by descents}
\author{Kassie Archer}
\address[K. Archer]{University of Texas at Tyler, Tyler, TX 75799 USA}
\email{karcher@uttyler.edu}

\author{Ira M. Gessel$^*$}
\address[I. Gessel]{Brandeis University, Waltham, MA 02453 USA}
\email{gessel@brandeis.edu}

\author{Christina Graves}
\address[C. Graves]{University of Texas at Tyler, Tyler, TX 75799 USA}
\email{cgraves@uttyler.edu}

\author{Xuming Liang}
\address[X. Liang]{Harvey Mudd College, Claremont, CA 91711 USA}
\email{eliang@g.hmc.edu}

\date{March 20, 2020}
\thanks{$^*$Supported by a grant from the Simons Foundation (\#427060, Ira Gessel)}

\begin{document}

\begin{abstract}

A descent of a labeled digraph is a directed edge $(s, t)$ with $s > t$. We count strong tournaments, strong digraphs,  acyclic digraphs, and forests by descents and edges. To count strong tournaments we use Eulerian generating functions and to count strong and acyclic digraphs we use a new type of generating function that we call a graphic Eulerian generating function.

\end{abstract}


\maketitle

\section{Introduction}

A \emph{digraph} $D$ consists of a finite vertex set $V$ together with a subset $E$ of $V\times V - \{\,(v,v): v\in V\,\}$, the set of edges of $D$. (We do not allow loops in our digraphs.) We call a digraph with vertex set $V$ a digraph \emph{on} $V$. We  assume that the vertices of our digraph are totally ordered, and for simplicity we take them to be integers. 

A \emph{descent} of a  digraph is an edge $(s,t)$ with $s>t$ and an \emph{ascent} is an edge $(s,t)$ with $s<t$. In this paper we count two important classes of digraphs, acyclic and strong, by edges and descents, generalizing the results of Robinson \cite{racyclic}. We also count strong tournaments and forests by descents.

A digraph is \emph{weakly connected} (or simply \emph{weak})
if its underlying graph is connected, and
is \emph{strongly connected} (or simply \emph{strong}) if for every two vertices $u$ and $v$ there is a directed path from $u$ to $v$ (allowing the empty path if $u=v$). The
\emph{weak} and \emph{strong components} of a digraph are the maximal weakly or
strongly connected subgraphs. Note that every edge is contained in a weak component but there may be edges not contained in any strong component. A \emph{source strong component} of a digraph is a strong component with no edges entering it from outside the component. (Robinson \cite{racyclic,rstrong} calls these \emph{out-components}.)  

We use the notation $[n]$ to denote the set $\{1,2,\dots,n\}$.
Given a digraph $D$ on $[n]$, we denote by $\e(D)$  the total number of edges of $D$, and by $\des(D)$  the number of descents of $D$. We define the \emph{descent polynomial} for a family of digraphs $\B_n$ on $[n]$ to be 
$$b_n(u) = \sum_{D \in \B_n} u^{\des(D)}.$$  
The coefficient of $u^k$ in $b_n(u)$ is the number of digraphs in $\B_n$ with exactly $k$ descents. Similarly, the \emph{descent-edge polynomial} $b_n(u,y)$ of $\B_n$ is defined as 
$$b_n(u,y) = \sum_{D \in \B_n} u^{\des(D)}y^{\e(D)}.$$ 
The coefficient of $u^k y^m$  in $b_n(u,y)$ is the number of digraphs in $\B_n$ with exactly $m$ edges, $k$ of which are descents.  
Note that $b_n(u) = b_n(u,1)$.

This paper is organized as follows. In Section~\ref{sec:digraphs}, we introduce several families of graphs and give known formulas for enumerating these families.   
In Section~\ref{sec:strong}, we enumerate strong tournaments  by the number of descents and enumerate strong digraphs by both the number of descents and edges. In Section~\ref{sec:acyclic}, we enumerate acyclic digraphs by the number of edges and descents; we also derive a formula for rooted trees and forests with a given number of descents and leaves.

\section{Families of digraphs}\label{sec:digraphs}
We are concerned primarily with four types of digraphs: strong tournaments, strong digraphs, acyclic digraphs, and trees.

\subsection{Strong tournaments} A digraph is a \emph{tournament} if there is exactly one directed edge between each pair of vertices. There are $2^{\binom{n}{2}}$ tournaments on $[n]$ since for any two vertices $u$ and $v$, a tournament contains the edge $(u,v)$ or $(v,u)$ but not both.
   
In \cite{mm}, Moon and Moser  found a formula for the probability that a randomly chosen tournament is strongly connected. Equivalently, they showed that the number $t_n$ of strong tournaments on $n \ge 1$ labeled vertices is given by the recurrence 
\begin{equation}
\label{e-mm1}
t_n= 2^{\binom{n}{2}} - \sum_{k=1}^{n-1} \binom{n}{k}2^{\binom{n-k}{2}} t_k.
\end{equation}
This recurrence is equivalent to the generating function relation
\begin{equation*}
\sum_{n=1}^\infty t_n \frac{x^n}{n!} = 1-\biggl(\sum_{n=0}^\infty 2^{\binom n2}\frac{x^n}{n!}\biggr)^{-1}.
\end{equation*}
The first few values of $t_n$ are $t_1=1, t_2=0, t_3=2, t_4=24, t_5=544,$ and $t_6=22320$.
In Section \ref{sec:strong tournaments} we generalize these formulas to count strong tournaments by descents, replacing the exponential generating functions with Eulerian generating functions.

\subsection{Strong digraphs}

Strong digraphs were first counted by Liskovets \cite{L69}, using a system of recurrences. Liskovets's recurrences were simplified by Wright \cite{wright}, who showed that the number $s_n$ of strong digraphs on $[n]$ is given by
\begin{equation}
\label{e-wright1}
s_n = \eta_n + \sum_{k=1}^{n-1} \binom{n-1}{k-1} s_k\eta_{n-k},
\end{equation}
where
\begin{equation}
\label{e-wright2}
\eta_n = 2^{n(n-1)} - \sum_{k=1}^{n-1} \binom{n}{k}2^{(n-1)(n-k)} \eta_k.
\end{equation}
The first few values of $s_n$ are $s_1=1, s_2=1, s_3=18, s_4=1606$, and $s_5=565080$.

A more direct approach to counting strong digraphs was given by Robinson \cite{racyclic} (see also \cite{rstrong}). Robinson's method will be discussed in more detail in Section \ref{SectionStrongDi}, and it is the basis for our approach to counting strong digraphs by edges and descents. We use a new kind of generating function that we call an \emph{Eulerian graphic generating function} whose properties are introduced in Section 
\ref{SectionEulerianGraphic}.

De Panafieu and Dovgal \cite{pd} have also counted acyclic and strong digraphs using an approach similar to Robinson's. 
A different approach to counting strong digraphs has been given by Ostroff \cite{ostroff}.

\subsection{Acyclic digraphs} 
An \emph{acyclic} digraph is a digraph with no directed cycles. That is, there is no nonempty directed path from any vertex  to itself. Robinson \cite{racyclic} 
showed that the number $a_n$ of acyclic digraphs on  $[n]$ 
is given by the  generating function
\[ \sum_{n=0}^\infty a_n\dpowg xn = \left( \sum_{n=0}^\infty  (-1)^n \dpowg xn\right)^{-1}.\]
The first few values of $a_n$ are $a_0=1, a_1=1, a_2=3, a_3=25,a_4=543,$ and $a_5=29281$.
(Robinson  had earlier \cite{racyclic1} found a different recurrence for counting acyclic digraphs that we will not discuss here.)

Using an approach similar to Robinson's, in Section~\ref{sec:acyclic digraphs} we generalize his formula to derive an Eulerian graphic generating function for 
counting acyclic digraphs by edges and descents. 

Another proof of Robinson's  formula was given by Stanley \cite{sacyclic}, using properties of chromatic polynomials of graphs. In Section \ref{sec:chromatic} we generalize Stanley's proof to count acyclic digraphs by edges and descents, using a generalization of the chromatic polynomial related to the chromatic quasisymmetric function of Shareshian and Wachs \cite{sw}.

\subsection{Trees} We consider a rooted tree  to be an acyclic digraph in which one vertex (the root) has outdegree 0 and every other vertex has outdegree 1. 
It is well known that there are $n^{n-1}$ rooted trees on $n$ vertices. More relevant to our results is that the exponential generating function $\sum_{n=1}^\infty n^{n-1}x^n/n!$ for rooted trees is the compositional inverse of $xe^{-x}$. In Section \ref{sec:trees} we 
use a variation of our first approach to counting acyclic digraphs to give a simple formula for the compositional inverse of the exponential generating function for counting rooted trees by descents.

\section{Strong Tournaments and Strong Digraphs}\label{sec:strong}

\subsection{Strong Tournaments}\label{sec:strong tournaments}

We first study strong tournaments, which are easier to enumerate than acyclic and strong digraphs. 
Since every tournament on $n$ vertices has exactly $\binom{n}{2}$ edges, we count strong tournaments by descents only. First, let us determine the descent polynomial for all tournaments on $[n]$.
For each pair of vertices $\{s,t\}$, exactly one of $(s,t)$ and $(t,s)$ is an edge. One of these edges is a  descent and the other is an ascent, and thus the descent polynomial for all tournaments on $[n]$ is 
$(1+u)^{\binom n2}$.
 
To count strong tournaments by descents, we will need some properties of the 
\emph{$q$-binomial coefficients} (also called \emph{Gaussian binomial coefficients}).
We first define the $q$-factorial $n!_q$ by 
\[ n!_q = 1\cdot(1+q)\cdots(1+q+ \cdots +q^{n-2})\cdot (1+q+\cdots +q^{n-1}), \]
with $0!_q=1$.
The  $q$-binomial coefficients, denoted 
  $\tqbinom ni$, are  defined by 
\[ \qbinom ni = \frac{n!_q}{i!_q(n-i)!_q}.\]
For $q=1$ they reduce to ordinary binomial coefficients. 

The $q$-binomial coefficients have several combinatorial interpretations, but the one that  we need is given in Lemma \ref{l-qbin} below. For disjoint sets of integers $S$ and $T$, we call an element $(s,t)$ of $S\times T$  a \emph{descent} if $s>t$ and an \emph{ascent} if $s<t$. Let $\des(S,T)$ denote the number of descents in $S\times T$.

\begin{lem}
\label{l-qbin}
For nonnegative integers $n$ and $i$, with $i\le n$, we have
\begin{equation*}
\qbinom ni=\sum_{(S,T)} q^{\des(S,T)},
\end{equation*}
where the sum is over all ordered partitions $(S,T)$ of $[n]$ for which $|S|=i$.
\end{lem}

This lemma is proved in \cite[Lemma 5.1]{qexp} by showing that the right side satisfies the same recurrence as the left side, $\tqbinom{n}{i} = q^i \tqbinom{n-1}i +\tqbinom{n-1}{i-1}$. It can also be derived easily from other well-known combinatorial interpretations for the $q$-binomial coefficients such as \cite[p.~56, Proposition 1.7.1]{ec1}.

We can now give a recurrence for the descent polynomial for strong tournaments on $[n]$.
 (Although the variable $q$ is traditionally used in $\tqbinom ni$, we will replace $q$ with $u$, as we are using the variable $u$ to weight descents.) Notice that setting $u=1$ gives the recurrence \eqref{e-mm1} of Moon and Moser.

\begin{thm}\label{StrongTournRec} Let $t_n(u)$ be the descent polynomial for the set of strong tournaments on $[n]$. Then for $n\ge1$ we have
\begin{equation}
\label{e-str}
 t_n(u) = (1+u)^{n \choose 2} - \sum_{k=1}^{n-1} \ubinom{n}{k} (1+u)^{n-k \choose 2}t_k(u).\end{equation}
\end{thm}

\begin{proof} 
Every nonempty tournament has a unique source strong component. Thus every tournament on $[n]$, for $n\ge 1$, can be constructed uniquely by choosing an ordered partition $(S,T)$ of $[n]$, with $S$ nonempty,  then constructing a strong tournament on $S$  and an arbitrary tournament on $T$, and adding all edges in $S\times T$. 
By Lemma \ref{l-qbin}, the contribution to the descent polynomial for all tournaments on $[n]$ with $|S|=k$ is $\tubinom{n}{k} t_k(u)\, (1+u)^{n-k \choose 2}$.
Thus
\begin{equation}
\label{e-str1}
(1+u)^{n \choose 2} = \sum_{k=1}^{n} \ubinom{n}{k} (1+u)^{n-k \choose 2}t_k(u).
\end{equation}
Solving for $t_n(u)$ gives \eqref{e-str}.
\end{proof}

The first few values of the polynomials $t_n(u)$ are 
\[t_1(u) =1,\:  t_2(u)= 0, \: t_3(u) = u + u^2, 
\: \text{ and } \: t_4(u) =  u+6{u}^{2}+10{u}^{3}+6{u}^{4}+{u}^{5}.
\]
Coefficients of $t_n(u)$ for larger $n$ can be easily computed from \eqref{e-str} and are given in Table~\ref{chartStrong}.
It is not difficult to show that $t_n(u)$ is a polynomial of degree $\binom n2 -1$ for $n\ge3$. Also, since reversing all the edges of a strong tournament gives another strong tournament, $t_n(u)$ is symmetric; i.e., $t_n(u) = u^{\binom n2}t_n(1/u)$.

\begin{table}
\begin{tabular}{c||r|r|r|r}
\diagbox{$d$}{$n$}   & \multicolumn{1}{c|}{4} & \multicolumn{1}{c|}{5} & \multicolumn{1}{c|}{6}  & \multicolumn{1}{c}{7}  \\ \hline \hline
1  & 1 & 1 & 1 & 1\\
2 &  6 & 13 & 22 & 33\\
3 &  10 & 56 & 172 & 402\\
4 &  6 & 123 & 717 & 2,674\\
5 &  1 & 158 & 1,910 & 11,614\\
6 &  0 & 123 & 3,547 & 36,293\\
7 &  0 & 56 & 4,791 & 86,305\\
8 &  0 & 13 & 4,791& 161,529 \\
9 & 0  & 1 & 3,547 & 242,890\\
10 & 0 & 0  & 1,910 & 297,003\\
11 & 0 & 0  & 717 & 297,003\\
12  & 0 & 0 &172 & 242,890\\
 13  & 0 & 0 & 22 &161,529\\
 14 &  0 & 0 & 1 & 86,305\\
 15 & 0&0&0& 36,293\\
  16 & 0&0&0& 11,614\\
   17 & 0&0&0& 2,674\\
    18 & 0&0&0& 402\\
     19 & 0&0&0& 33\\
      20 & 0&0&0& 1\\
 \hline
 TOTAL & 24 & 544 & 22,320 & 1,677,488
\end{tabular}
\caption{The number of strong tournaments on $n$ vertices with $d$ descents for $4 \leq n \leq 7$; equivalently, the coefficients of $u^d$ in $t_n(u)$.}
\label{chartStrong}
\end{table}

The next result gives a divisibility property for $t_n(u)$.

\begin{prop}
The polynomial $t_n(u)$ is divisible by  $(1+u)^{\floor{n/2}}$.
\end{prop}
\begin{proof}
Let $v_n(u)=t_n(u)/(1+u)^{\floor{n/2}}$. Then from \eqref{e-str} we obtain the recurrence
\begin{equation*}
v_n(u) = (1+u)^{\binom n2 -\floor{n/2}}-\sum_{k=1}^{n-1}\ubinom{n}{k} (1+u)^{\binom{n-k}{2}-\floor{n/2}+\floor{k/2}}v_k(u)
\end{equation*}
for $n\ge1$.
It is easy to check that  $\binom n2 -\floor{n/2}\ge0$, so
it suffices to show that the expression multiplied by $v_k(u)$ in the sum on the right  is a polynomial in $u$.

Let
\begin{equation*}
E(n,k) =\binom{n-k}{2}-\floor{\frac n2}+\floor{\frac k2}.
\end{equation*}
We first show that $E(n,k)\ge 0$ for $k<n-1$. Note that 
\begin{align*}
E(n,k) \geq  \binom{n-k}{2}-\frac n2+\frac {k-1}{2} \\
= \frac{(n-k)(n-k-2) -1}{2}.
\end{align*}
For $k \leq n-2$, this gives $E(n,k) \geq -1/2$, so since $E(n,k)$ is an integer,
we have $E(n,k)\ge0$.  In the case where $k=n-1$, we must consider the parity of $n$. If $n$ is odd, then $E(n,n-1) = 0$. However, if $n$ is even, $E(n,n-1) = -1$. To complete the proof it suffices to show that if $n$ is even then $\tubinom {n}{n-1}$ is divisible by $1+u$. But if $n$ is even then
\begin{equation*}
\ubinom{n}{n-1}= 1+u+u^2+\cdots +u^{n-1}=(1+u)(1+u^2+u^4+\cdots+u^{n-2}),
\end{equation*}
thus completing the proof.
\end{proof}

The coefficients of $t_n(u)/(1+u)^{\lfloor n/2\rfloor}$ seem to be nonnegative, but we are not able to prove this.

The recurrence of Theorem \ref{StrongTournRec} can also be expressed with generating functions.
An \emph{Eulerian generating function} \cite[p.~321, Example 3.18.1c]{ec1} is a generating function of the form
\begin{equation*}
\sumz n  a_n\dpowq xn.
\end{equation*}
Note that for $q=1$, this reduces to an exponential generating function.

The multiplication of Eulerian generating functions is similar to that of exponential generating functions. 
If
\begin{equation*}
a(x) = \sumz n  a_n \dpowq xn \quad \text{ and } \quad
b(x) = \sumz n  b_n \dpowq xn,
\end{equation*}
then
\begin{equation*}
a(x)b(x) = \sumz n c_n \dpowq xn, 
\end{equation*}
where the coefficient $c_n$ is given by
\begin{equation*}
c_n = \sum_{i=0}^n \qbinom ni a_i b_{n-i}.
\end{equation*}

The generating function for $t_n(u)$ can be derived directly from the formula given in Theorem \ref{StrongTournRec} and the multiplication property of Eulerian generating functions:
\begin{cor}
\label{StrongTournGen} 
Let $T(x) = \sum_{n=1}^\infty t_n(u) x^n/n!_u$ be the Eulerian generating function for strong tournaments by descents 
and let $U(x) = \sum_{n=0}^\infty (1+u)^{\binom n2}x^n/n!_u$
be the Eulerian generating function for all tournaments by descents.
Then 
\begin{equation*}
\label{e-stourn}
T(x) =1-U(x)^{-1}
\end{equation*}
and 
\begin{equation}
\label{e-stourn2}
U(x) = \frac{1}{1-T(x)}.
\end{equation}
\end{cor}
\begin{proof}
Equation \eqref{e-str1} is equivalent to $U(x) =1+T(x)U(x)$ from which the two formulas follow easily.
\end{proof}

Equation \eqref{e-stourn2} has a simple combinatorial interpretation obtained by iterating the decomposition described in the proof of Theorem \ref{StrongTournRec}: every tournament may be decomposed into a sequence of strong tournaments with all edges between the strong tournaments oriented from left to right.

\subsection{Eulerian Graphic Generating Functions}
\label{SectionEulerianGraphic}

In this section, we introduce a new type of generating function which will be useful in enumerating both acyclic and strong digraphs by  descents and edges. This new generating function  is a generalization of a \emph{graphic generating function} (also called a \emph{special generating function} \cite{racyclic}
or \emph{chromatic generating function}
\cite[p.~321, Example 3.18.1c]{ec1}),
which is a generating function of the form
\begin{equation*}
\sumz n a_n \dpowgy xn,
\end{equation*}
often with $y=1$. Graphic generating functions were first used by Robinson \cite{racyclic} and by Read \cite{read} (in the case $y=1$, with a slightly different normalization); further applications of graphic generating functions have been given by 
Gessel and Sagan \cite{GS96}, Gessel \cite{decomp}, and de Panafieu and Dovgal \cite{pd}.

We define an \emph{Eulerian graphic generating function} to be a generating function of the form
\begin{equation}
\label{e-Eggf}
 \sumz n a_n  \dpowbq{x}{n}.
\end{equation}
Given two  Eulerian graphic generating functions $a(x)$ and $b(x)$ defined by
\[ a(x) = \sumz n a_n  \dpowbq{x}{n}\quad \text{and} \quad b(x) = \sumz n b_n \dpowbq{x}{n},\]
we multiply them to obtain
\[ a(x)b(x) = \sumz n c_n \dpowbq{x}{n} \]
where\[ c_n = \sum_{i=0}^n \qbinom ni (1+y)^{i(n-i)} a_i b_{n-i}. \]

In all of our formulas from here on we will modify the Eulerian graphic generating functions by taking $q=(1+uy)/(1+y)$. The combinatorial interpretation of these modified Eulerian graphic generating functions is explained by the following lemma.

\begin{lem}
\label{l-qbinom}
Let $q=(1+uy)/(1+y)$. Then $\tqbinom ni (1+y)^{i(n-i)}$ is a polynomial in $u$ and $y$, and the coefficient of $u^j y^m$ in  $\tqbinom ni (1+y)^{i(n-i)}$ is the number of ordered pairs $(S,A)$ where $S$ is an $i$-subset of $[n]$ and $A$ is an $m$-subset of $S\times ([n]-S)$ containing exactly $j$ descents.
\end{lem}

\begin{proof}
Let $S$ be an $i$-subset of $[n]$ such that $S\times([n]-S)$ has $k$ descents, and thus $i(n-i)-k$ ordered pairs that are not descents. Define the weight of a subset $A\subseteq S\times ([n]-S)$ to be $u^{\des(A)}y^{|A|}$. 
To count such weighted subsets of $S\times([n]-S)$ we specify $A$ by deciding which descents and ascents of $S\times([n]-S)$ are included in $A$.
Each descent in $S\times ([n]-S)$ can either be included in $A$, contributing a factor of $uy$ to the weight of $A$, or excluded, contributing a factor of 1. Similarly, each ascent in $S\times ([n]-S)$ can either be included in $A$, contributing a factor of $y$, or excluded, contributing a factor of 1. Thus the sum of the weights of all $A\subseteq S \times ([n]-S)$ is $(1+uy)^k (1+y)^{i(n-i) -k}$.

Now define $Q_{n,i,k}$ by 
\begin{equation*}
\qbinom ni = \sum_{k=0}^{i(n-i)} Q_{n,i,k}q^k.
\end{equation*}
Then by Lemma \ref{l-qbin},  $Q_{n,i,k}$ is the number of $i$-subsets $S$ of $[n]$ such that $S\times([n]-S)$ has $k$ descents.  Thus the sum over all $i$-subsets $S\subseteq[n]$ of the weights of all $A\subseteq S \times ([n]-S)$ is 
\begin{align*}\sum_{k=0}^{i(n-i)} Q_{n,i,k} (1+uy)^k (1+y)^{i(n-i)-k}
 &=(1+y)^{i(n-i)}\sum_{k=0}^{i(n-i)} Q_{n,i,k} \left(\frac{1+uy}{1+y}\right)^k\\
 &=(1+y)^{i(n-i)}\qbinom ni,
\end{align*}
where $q=(1+uy)/(1+y)$.
\end{proof}

The modified Eulerian graphic generating functions may be viewed another way. 
Note that 
\begin{equation*}
n!_q(1+y)^{\binom n2}=\prod_{i=1}^{n} (1+q+\cdots + q^{i-1}) (1+y)^{i-1}.
\end{equation*}
Setting $q=(1+uy)/(1+y)$ and letting $P(i)$ denote the $i$th factor in this product gives 
\begin{align*}
P(i)&:=
(1+y)^{i-1}+(1+uy)(1+y)^{i-2}+(1+uy)^2(1+y)^{i-3}+\cdots+(1+uy)^{i-1}\\
   &\phantom{:}=\frac{(1+y)^i-(1+uy)^i}{y(1-u)}.
\end{align*}
So if we let $F(n) = P(1)P(2)\cdots P(n)$, then the modified Eulerian graphic generating functions are of the form $\sumz n a_n x^n/F(n)$.

Note that if we set $u=1$, then $q=(1+uy)/(1+y)$ becomes 1. So in this case the 
Eulerian graphic generating function \eqref{e-Eggf} reduces to the ``ordinary" graphic generating function
\begin{equation*}
\sumz n a_n \dpowgy xn.
\end{equation*}

\subsection{Strong Digraphs}
\label{SectionStrongDi}

We now find a generating function for the descent-edge polynomial for the set of strong digraphs. 
Before beginning the proof, we need several preliminary definitions.  Let $\Delta$ be the linear transformation that converts an exponential generating function to an Eulerian graphic generating function. That is,
\[
\Delta\left(\sum_{n=0}^\infty a_n \frac{x^n}{n!}\right) = \sum_{n=0}^{\infty} a_n \dpowbq xn.
\]

Let $G(x)$ be the Eulerian graphic generating function for the descent-edge polynomials of all digraphs.  
To specify a digraph $D$ on $[n]$, for each possible edge $(s,t)$ we either include it in $D$ or exclude it. If $s>t$ then including $(s,t)$ as an edge contributes a factor $uy$ to the descent-edge weight of $D$ and excluding it contributes a factor of 1. Similarly, if $s<t$ then including $(s,t)$ as an edge contributes a factor $y$ to the descent-edge weight of $D$ and excluding it contributes a factor of 1. Since there are $\binom n2$ possible edges $(s,t)$ with $s>t$ and $\binom n2$  with $s<t$, 
the descent-edge polynomial of the set of all digraphs on $[n]$ is $(1+uy)^{n \choose 2}(1+y)^{n \choose 2}$
and therefore
\begin{equation}
\label{e-G(x)}
G(x) =\sum_{n=0}^{\infty} (1+uy)^{n \choose 2}   (1+y)^{n \choose 2}\dpowbq xn  =\sum_{n=0}^{\infty} (1+uy)^{n \choose 2} \frac{x^n}{n!_q}.
\end{equation}

Now let 
\[d_n(u,y;\beta)= \sum_D u^{\des(D)}y^{\e(D)}\beta^{\ssc(D)},\]
where the sum is over all digraphs $D$ on
the vertex set $[n]$,  $\e(D)$ is the number of edges of
$D$, and $\ssc(D)$ is the number of source strong components of $D$. As we have just seen,
\begin{equation}
\label{e-dn1}
d_n(u,y;1) = (1+uy)^{\binom n2}(1+y)^{\binom n2}.
\end{equation}

Let $s_n(u,y)$ be the descent-edge polynomial for the set of strong digraphs on $n$ vertices, and let 
\[S(x)=\sum_{n=1}^\infty s_n(u,y)\dpow xn.\] 
Define polynomials $v_n(u,y;\beta)$ by
\begin{equation}
\label{e-vn}
e^{\beta S(x)}=\sumz n v_n(u,y;\beta) \dpow xn.
\end{equation}
Then by  the ``exponential formula'' \cite[p.~5, Corollary 5.1.6]{ec2}, $v_n(u,y;\beta)$ is the descent-edge polynomial for digraphs on $n$ in which every weak component is strong, where each weak component is weighted $\beta$. The Eulerian graphic generating function for the polynomials $v_n(u,y;\beta)$ is thus $\Delta(e^{\beta S(x)})$.
We can now count strong digraphs by edges and descents, generalizing the result of Robinson who proved the case $u=1$ (and thus $q=1$) of the next result.

\begin{thm}
\label{t-strong}
Let $S(x)$ be the exponential generating function for the descent-edge polynomial for strong digraphs and let $G(x)$ be the Eulerian graphic generating function for all digraphs, given in \eqref{e-G(x)}.
Then
\begin{equation}
\label{e-strong}
S(x) = -\log \bigl(\Delta^{-1}(G(x)^{-1})\bigr).
\end{equation}
\end{thm}

\begin{proof}
We will count in two ways ordered pairs $(D, C)$ where 
$D$ is a digraph on $[n]$ and $C$ is a subset of the set of source strong components of $D$. (We may identify $C$ with the digraph whose weakly connected components are the elements of the set $C$.)
  To such a pair we assign the weight $\beta^{|C|} u^{\des(D)} y^{\e(D)}$. We compute the sum of the weights of these pairs in two ways.

First, we may choose $D$ as an arbitrary digraph on $[n]$ and then choose $C$ as an arbitrary subset of the source strong components of $D$. Thus the sum of the weights is $d_n(u,y;\beta+1)$. 

Alternatively, we may  count pairs $(D,C)$ by first choosing a subset $T$ of $[n]$, constructing a set of strong digraphs $C$ on $T$, choosing a digraph $D'$ on $[n]-T$ and choosing a subset $E$ of $T\times( [n]-T)$. We then construct $D$ by adding to $D'$ the digraphs in $C$ together with the elements of $E$ as edges. Then $\des(D) = \des(E)+\des(C)+\des(D')$ and 
$\e(D) = |E| +\e(C)+ \e(D')$. It follows from Lemma \ref{l-qbinom}  that the sum of the weights of the pairs $(D,C)$ in which $C$ has a total of $i$ vertices is $(1+y)^{i(n-i)}\tqbinom{n}{i}v_i(u,y; \beta) d_{n-i}(u,y;1)$. Summing over $i$ and using  \eqref{e-dn1} gives
\begin{equation}
\label{e-vrec}
\sum_{i=0}^n (1+y)^{i(n-i)}\qbinom{n}{i}v_i(u,y;\beta) 
    d_{n-i}(u,y;1)=d_n(u,y;\beta+1),
\end{equation}
which is equivalent by \eqref{e-vn} to 
\begin{equation*}
\label{e-Dalpha}
\Delta(e^{\beta S(x)}) G(x)= \sumz n d_n(u,y;\beta+1) \dpowbq xn,
\end{equation*}
where $G(x)$ is given by \eqref{e-G(x)}. Now we set $\beta=-1$. Since $d_n(u,y;0)=0$ for $n>0$ we obtain
\begin{equation*}
\Delta(e^{- S(x)}) G(x)= 1.
\end{equation*}
Solving for $S(x)$ yields \eqref{e-strong}.
\end{proof}

We can now give extensions of Wright's recurrences \eqref{e-wright1} and \eqref{e-wright2} for the polynomials $s_n(u,y)$.
\begin{cor}\label{StrongDiRec}
The descent-edge polynomial for strong digraphs on $n$ vertices $s_n(u,y)$ satisfies the recurrence
\begin{equation}
\label{e-s-eta}
 s_n(u,y) = \eta_n(u,y) + \sum_{k=1}^{n-1}  \binom{n-1}{k-1} s_k(u,y)\eta_{n-k}(u,y), \  n\ge 1,
\end{equation}
where the polynomials $\eta_n(u,y)$ are determined by 
\begin{multline}
\label{e-etarec}
\qquad
\eta_n(u,y) = (1+y)^{\binom n2}(1+uy)^{\binom n2}\\
  -\sum_{k=1}^{n-1} \qbinom{n}{k}  (1+uy)^{\binom{n-k}{2}}
  (1+y)^{(n-k)(n+k-1)/2} \eta_k(u,y),
\qquad
\end{multline}
with $q = (1+uy)/(1+y)$.
\end{cor}

\begin{proof}
Let $E(x) = 1-e^{-S(x)}$, so
\begin{equation}
\label{e-SE}
S(x) = \log\frac{1}{1-E(x)},
\end{equation}
and define polynomials $\eta_n(u,y)$ by
$E(x) = \sum_{n=1}^\infty \eta_n(u,y) x^n/n!$. Thus for $n\ge1$, 
$\eta_n(u,y) = -v_n(u,y;-1)$, where $v_n(u,y;\beta)$ is defined in \eqref{e-vn}.
Then \eqref{e-etarec} is obtained by rearranging the case $\beta=-1$ of \eqref{e-vrec}, using \eqref{e-dn1} and $d_n(u,y;0)=0$ for $n>0$.

Differentiating \eqref{e-SE} with respect to $x$ and simplifying gives
\begin{equation*}
S'(x) = E'(x)+S'(x)E(x).
\end{equation*}
Equating coefficients of $x^{n-1}/(n-1)!$ gives \eqref{e-s-eta}.
\end{proof}

The first few values of the polynomials $s_n(u, y)$ are $s_1(u,y) =1, s_2(u,y)= uy^2$, and
\[ s_3(u,y) = uy^3 + u^2y^3 + uy^4 + 7u^2y^4 + u^3y^4 + 3u^2y^5+3u^3y^5 + u^3y^6.\] 
The values of $s_n(u,y)$ for larger $n$ can be easily computed from the recurrences of Corollary~\ref{StrongDiRec}. We provide the values of $s_n(u,1)$ and $s_n(1,y)$ for small values in Tables~\ref{NumberStrong}A and \ref{NumberStrong}B.

\begin{table}
\begin{tabular}{ccc}
\subcaptionbox{Number of strong digraphs on $n$ vertices with $d$ descents}{
\begin{tabular}{c||r|r|r|r}
\diagbox{$d$}{$n$}  & \multicolumn{1}{c|}{3}   & \multicolumn{1}{c|}{4} & \multicolumn{1}{c|}{5} & \multicolumn{1}{c}{6}    \\ \hline \hline
1  & 2&10 & 122 & 3,346\\
2 &  11&154 & 3,418 & 142,760\\
3 &  5&540 & 27,304 & 1,938,178\\
4 &  0&581 & 90,277 & 12,186,976\\
5 &  0&272 & 150,948 & 42,696,630\\
6 &  0&49 & 150,519 & 94,605,036 \\
7 &  0&0 &  95,088 & 145,009,210\\
8 &  0&0 &  37,797& 161,845,163 \\
9 & 0 &0 &  8,714 & 134,933,733\\
10 & 0&0 & 893 & 84,656,743\\
11 & 0 &0& 0  & 39,632,149\\
12  & 0 & 0&0  & 13,481,441\\
 13  & 0 & 0&0 &3,156,845\\
 14 &  0 & 0 &0& 455,917\\
 15 & 0&0& 0&30,649\\
 \hline
 TOTAL & 18 &1,606 & 565,080 & 734,774,776\\
\multicolumn{5}{c}{} \\
\multicolumn{5}{c}{} \\
\multicolumn{5}{c}{} \\
\end{tabular} }
& \quad &
\subcaptionbox{Number of strong digraphs on $n$ vertices with $e$ edges}{%
\begin{tabular}{c||r|r|r}
\diagbox{$e$}{$n$}   & \multicolumn{1}{c|}{3} & \multicolumn{1}{c|}{4} & \multicolumn{1}{c}{5}    \\ \hline \hline
3 &  2 & 0 & 0\\
4 &  9 & 6 & 0\\
5 &  6 & 84 & 24\\
6 &  1 & 316 & 720 \\
7 &  0 &  492 & 6,440\\
8 &  0 &  417& 26,875 \\
9 & 0  &  212 & 65,280\\
10 & 0 & 66 & 105,566\\
11 & 0 & 12  & 122,580\\
12  & 0 & 1  & 106,825\\
 13  & 0 & 0 & 71,700\\
 14 &  0 & 0 & 37,540\\
 15 & 0&0& 15,344\\
 16 & 0 & 0 & 4,835\\
 17 & 0 & 0 & 1,140\\
 18 & 0 & 0 & 190\\
 19 & 0 & 0 & 20\\
 20 & 0 & 0 & 1\\
 \hline
 TOTAL & 18 & 1,606 & 565,080
\end{tabular} } 
\end{tabular}
\caption{Number of strong digraphs by descents and edges}
\label{NumberStrong}
\end{table}

Note that setting $u=y=1$ in the recurrences of Corollary \ref{StrongDiRec} gives Wright's recurrences \eqref{e-wright1} and \eqref{e-wright2}, so Wright's $\eta_n$ is our $\eta_n(1,1)$. In fact, Wright also knew the corresponding recurrences for $s_n(1,y)$ and $\eta_n(1,y)$, counting strong digraphs by edges. He also stated that \eqref{e-wright1} and \eqref{e-wright2} look as if they should possess combinatorial interpretations, but that he was not able to find one. 
He wrote, ``We can show that $\eta_n$ is non-negative, though $\eta_2 = 0$. But some of the coefficients in the polynomials  $\eta_n(y)$ are negative and this makes it seem somewhat unlikely that $\eta_n$ has a simple combinatorial meaning." (Wright's $\eta_n(y)$ is our $\eta_n(1,y)$.) 

Despite Wright's pessimism, $\eta_n$ does have a simple combinatorial interpretation, which suggests a connection between the enumeration of strong tournaments and the enumeration of strong digraphs. 
If we multiply Moon and Moser's recurrence \eqref{e-mm1} for strong tournaments by $2^{\binom n2}$, we get
\begin{equation*}
2^{\binom n2}t_n= 2^{n(n-1)} - \sum_{k=1}^{n-1} \binom{n}{k}2^{(n-1)(n-k)}\cdot 2^{\binom k2}t_k.
\end{equation*}
Comparing with \eqref{e-wright2}, we see that Wright's $\eta_n$ is equal to our $2^{\binom n2}t_n$.

Thus \eqref{e-SE} for $u=y=1$ may be written
\begin{equation}
\label{e-wright3}
\sum_{n=1}^\infty s_n \dpow xn 
  =-\log\biggl(1-\sum_{n=1}^\infty 2^{\binom n2}t_n\dpow xn \biggr).
\end{equation}

Although the coefficients of $\eta_n(u,y)$ are not in general nonnegative, we can derive a one-parameter refinement of the formula $\eta_n = 2^{\binom n2}t_n$ with nonnegative coefficients from Theorems \ref{StrongTournRec} and \ref{t-strong}. Note that $\eta_2=\eta_2(1,1)=0$ and $\eta_2(u,y) = -1+uy^2$. This suggests that if we want a specialization of $
\eta_n(u,y)$ with nonnegative coefficients, we might try setting $u=y^{-2}$.

\begin{prop}
The polynomials $\eta_n(u,y)$ defined by \eqref{e-etarec} and the descent polynomials  for strong tournaments by descents $t_n(u)$, determined by \eqref{e-str}, are related by 
\begin{equation}
\label{e-eta-t}
\eta_n(y^{-2},y) = (1+y)^{\binom n2}t_n(y\inverse).
\end{equation}
\end{prop}

\begin{proof}
We show that both sides of \eqref{e-eta-t} satisfy the same recurrence. (This recurrence does not require any initial values.)
If $u=y^{-2}$ then $1+uy=1+y^{-1}$ and $q=(1+uy)/(1+y)=y^{-1}$, so setting $u=y^{-2}$ in  \eqref{e-etarec} gives
\begin{multline*}
\qquad
\eta_n(y^{-2},y) = (1+y)^{\binom n2}(1+y\inverse)^{\binom n2}\\
  -\sum_{k=1}^{n-1} \binom{n}{k}_{\!\!y\inverse}  (1+y\inverse)^{\binom{n-k}{2}}
  (1+y)^{(n-k)(n+k-1)/2} \eta_k(y^{-2},y).
\qquad
\end{multline*}

Setting $u=y^{-1}$ in \eqref{e-str}, multiplying by $(1+y)^{\binom n2}$, and simplifying gives
\begin{multline*}
\quad
(1+y)^{\binom n2}t_n(y\inverse)=(1+y)^{\binom n2}(1+y\inverse)^{\binom n2}\\
 -\sum_{k=1}^{n-1} \binom {n}{k}_{\!\!y\inverse} (1+y)^{\binom n2}(1+y\inverse)^{\binom{n-k}2} (1+y)^{-\binom k2}\cdot (1+y)^{\binom k2}t_k(y\inverse).
 \quad
\end{multline*}
Then \eqref{e-eta-t} follows by comparing these two recurrences and using  $(1+y)^{(n-k)(n+k-1)/2}=
(1+y)^{\binom n2}(1+y)^{-\binom k2}$.
\end{proof}

Applying \eqref{e-eta-t} to \eqref{e-SE} gives
\begin{equation}
\label{e-sdt}
\sumz n s_n(y^{-2},y) \dpow xn 
   = -\log\biggl(1-\sum_{n=1}^\infty (1+y)^{\binom n2}t_n(y^{-1})\dpow xn \biggr).
\end{equation}

If two exponential generating functions $f$ and $g$ are related by $f=-\log(1-g)=\sum_{n=1}^\infty (n-1)!\,g^n\!/n!$ then $f$ may be interpreted as counting cycles of the objects counted by $g$, so we might hope that \eqref{e-wright3} and \eqref{e-sdt} could be explained combinatorially by a bijection from strong digraphs to cycles of strong tournaments with some additional structure. But we have not been able to find such a bijection.

\section{Acyclic Digraphs and Trees}\label{sec:acyclic}

\subsection{Acyclic Digraphs}\label{sec:acyclic digraphs}
We begin this section by enumerating acyclic digraphs by their number of edges, descents, and sources. We again make use of Eulerian graphic generating functions, and also follow closely the proof in \cite{G96} for enumeration of acyclic digraphs by sources and edges (which is based on Robinson's proof \cite{racyclic}).

Let 
\[a_n(u,y;\beta)= \sum_D u^{\des(D)}y^{\e(D)}\beta^{\source(D)},\]
 where the sum is over all acyclic digraphs $D$ on
the vertex set $[n]$,  $\e(D)$ is the number of edges of
$D$, and $\source(D)$ is the number of sources of $D$; that is, the number of
vertices of $D$ of in-degree 0. Let $a_n(u,y)=a_n(u,y;1)$.

To count acyclic digraphs by sources we take an acyclic digraph and add
some new vertices as sources. The new vertices
will be a subset of the set of sources of the expanded digraph. This gives a formula expressing $a_n(u,y; \beta+1)$ in terms of $a_j(u,y)$ for $j\le n$. Since every nonempty acyclic digraph has at least one source, the formula for $a_n(u,y;0)$ gives a recurrence for $a_n(u,y)$.

\begin{lem}
\label{l-acyclic}
For every nonnegative integer $n$, we have
\begin{equation}
\label{e-ac-rec}
\sum_{i=0}^n \qbinom{n}{i} (1+y)^{i(n-i)}\beta^i a_{n-i}(u,y)=a_n(u,y;\beta+1).
\end{equation}
\end{lem}

\begin{proof}
We count ordered pairs $(D,C)$, where $D$ is an acyclic digraph on $[n]$ and $C$ is a subset of the set of sources of $D$. To such a pair we assign the weight $u^{\des(D)} y^{\e(D)}\beta^{|C|}$.  We compute the sum of the weights of these pairs in two ways.

First, we may choose $D$ as an acyclic digraph on $[n]$ and then choose $C$ as an arbitrary subset of the sources of $D$. Thus the sum of the weights is $a_n(u,y;\beta+1)$. 

We may also count pairs $(D,C)$ by first choosing a subset $C$ of $[n]$, choosing a digraph $D'$ on $[n]-C$ and choosing a subset $E$ of $C\times ([n]-C)$. We then construct $D$ by adding to $D'$ the elements of $C$ as vertices and the elements of $E$ as edges. Then $\des(D) = \des(E)+\des(D')$ and 
$\e(D) = |E| + \e(D')$. Then it follows from Lemma \ref{l-qbinom}  that the sum of the weights of the pairs $(D,C)$ in which $|C|=i$ is $\tqbinom{n}{i}(1+y)^{i(n-i)}\beta^i a_{n-i}(u,y)$, and summing on $i$ gives the left side of \eqref{e-ac-rec}.
\end{proof}

\begin{thm} 
\label{t-acyclic}
Let $a_n(u,y)$ be the descent-edge polynomial for the set of acyclic digraphs on $n$ vertices and let $A(x)$ be the Eulerian graphic generating function for $a_n(u,y)$ where $q = (1+uy)/(1+y)$.  Then
\begin{equation}
\label{e-A(x)}
A(x) = \left(\sumz n (-1)^n \dpowbq xn \right)^{-1}.
\end{equation}

More generally, the Eulerian graphic generating function for $a_n(u,y;\beta)$ is 
\begin{equation}
\label{e-A(x,a)}
\left(\sumz n  (\beta -1)^n\dpowbq xn \right)\biggm/
   \left(\sumz n (-1)^n \dpowbq xn \right)\qedhere.
\end{equation}

\end{thm}
\begin{proof}
Equation \eqref{e-ac-rec} is equivalent to 
\begin{equation}
\label{e-A2}
\sumz n a_n(u, y; \beta+1) \dpowbq xn = 
  \left(\sumz n \beta ^n \dpowbq xn \right)A(x).
\end{equation}
Setting $\beta=-1$ in \eqref{e-A2}, and using the fact that $a_n(u,y;0) = 0$ for $n>0$ gives \eqref{e-A(x)}. Then replacing $\beta$ by $\beta-1$ in \eqref{e-A2} and applying \eqref{e-A(x)} gives \eqref{e-A(x,a)}.
\end{proof}

An interesting special case of \eqref{e-A(x,a)} is obtained by setting $u=0$, so that we are counting (acyclic) digraphs with no descents by the number of sources. We find that 
\begin{equation*}
a_n(0, y; \beta) = \prod_{i=0}^{n-1}\bigl( \beta + (1+y)^i-1\bigr).
\end{equation*}
This is not difficult to prove directly: since $a_0(0,y;\beta)=1$, it is enough to show that for $n>0$ we have 
\begin{equation}
\label{e-a0}
a_n(0, y; \beta) = a_{n-1}(0, y; \beta)\bigl( \beta + (1+y)^n-1\bigr).
\end{equation}
To prove \eqref{e-a0}, we note that every acyclic digraph on $[n]$ with no descents is obtained from an acyclic digraph on $[n-1]$ with no descents by adding $n$ as a vertex, together with some of the edges $(i,n)$ for $i\in [n-1]$. If none of these edges are added then $n$ is a source; otherwise, $n$ is not a source. Equation \eqref{e-a0} follows immediately from this construction.

From either \eqref{e-ac-rec} or \eqref{e-A(x)} we obtain a recurrence for $a_n(u,y)$:
\begin{cor}\label{AcyclicRec} Let $a_n(u,y)$ be the descent-edge polynomial for the set of acyclic digraphs on $n$ vertices. Then
\[ a_n(u,y) = \sum_{i=0}^{n-1} (-1)^{n-i-1}\qbinom{n}{i} (1+y)^{i(n-i)}a_{i}(u,y) \]
where $q = (1+uy)/(1+y)$.\qed
\end{cor}

The polynomials $a_n(u, y)$ for the first few values of $n$ are given by  $a_1(u,y) =1,$ $a_2(u,y)= 1 + y + uy$, and
\[ a_3(u,y) = 1 + (3 + 3u)y + (3 + 6u + 3u^2)y^2 + (1 + 2u + 2u^2 + u^3)y^3.\] Coefficients of $a_n(u,y)$ for larger $n$ can be computed from the formula in Corollary~{\ref{AcyclicRec}}; we provide the values of $a_n(u,1)$ for small values of $n$ in Table~\ref{NumberAcyclicDescent}.

\begin{table}
\begin{tabular}{c||r|r|r|r|r|r|r}
\diagbox{$u$}{$n$} & 1 & 2 & \multicolumn{1}{c|}{3} & \multicolumn{1}{c|}{4} & \multicolumn{1}{c|}{5} & \multicolumn{1}{c|}{6} & \multicolumn{1}{c|}{7}  \\ \hline \hline
0 & 1 & 2 & 8 & 64 & 1,024 & 32,768 & 2,097,152\\
1 & 0 & 1 & 11 & 161 & 3,927 & 172,665 & 14,208,231\\
2 & 0 & 0 & 5 & 167 & 6,698 & 419,364& 45,263,175\\
3 & 0 & 0 & 1 & 102 & 7,185 & 656,733& 94,040,848\\
4 & 0 & 0 & 0 & 39 & 5,477 & 757,939&145,990,526\\
5 & 0 & 0 & 0 & 9 & 3,107 & 686,425& 181,444,276\\
6 & 0 & 0 & 0 & 1 & 1,329 & 504,084& 187,742,937\\
7 & 0 & 0 & 0 & 0 & 423 & 305,207& 165,596,535\\
8 & 0 & 0 & 0 & 0 & 96 & 153,333& 126,344,492\\
9 & 0 & 0 & 0 & 0 & 14 & 63,789& 84,115,442\\
10 & 0 & 0 & 0 & 0 & 1 & 21,752& 49,085,984\\
11 & 0 & 0 & 0 & 0 & 0 & 5,959& 25,134,230\\
12 & 0 & 0 & 0 & 0 & 0 &1,267& 11,270,307\\
 13 & 0 & 0 & 0 & 0 & 0 & 197 & 4,403,313\\
 14 & 0 & 0 & 0 & 0 & 0 & 20& 1,486,423\\
 15 & 0 & 0 & 0 & 0 & 0 & 1& 428,139\\
 16 & 0 & 0 & 0 & 0 & 0 & 0 & 103,345\\
 17 & 0 & 0 & 0 & 0 & 0 & 0 & 20,369 \\
 18 & 0 & 0 & 0 & 0 & 0 & 0 & 3,153 \\
 19 & 0 & 0 & 0 & 0 & 0 & 0 & 360 \\
 20 & 0 & 0 & 0 & 0 & 0 & 0 & 27 \\
 21 & 0 & 0 & 0 & 0 & 0 & 0 & 1 \\ 
 \hline
 TOTAL & 1 & 3 & 25 & 543 & 29,281 & 3,781,503 & 1,138,779,265
\end{tabular}
\caption{Values of $a_{n}(u,1)$, the number of acyclic digraphs on $n$ vertices with $u$ descents, for $n \leq 7$.  The total is the number of labeled acyclic digraphs on $n$ vertices.}
\label{NumberAcyclicDescent}
\end{table}

Robinson \cite{racyclic, rstrong} gave a common generalization of the case $u=1$ of Theorems \ref{t-acyclic} and \ref{t-strong}. Given a class $\mathscr{S}$ of strong digraphs he found a relation between the generating function for digraphs in $\mathscr{S}$ and the generating function for digraphs all of whose strong components are in $\mathscr{S}$. If $\mathscr{S}$ is the class of all strong digraphs we get the case $u=1$ of Theorem \ref{t-strong} and if $\mathscr{S}$ is the class of 1-vertex graphs we get the case $u=1$ of Theorem \ref{t-acyclic}. Our approach could be applied to extend Robinson's generalization to include descents.

\subsection{Chromatic polynomials}\label{sec:chromatic}

Stanley \cite{sacyclic} derived the generating function for acyclic digraphs from his theorem \cite[Corollary 1.3]{sacyclic} that if $\chi_G(\lambda)$ is the chromatic polynomial of a graph $G$ with $n$ vertices, then the number of acyclic orientations of $G$ is $(-1)^n\chi_G(-1)$. We can use his approach to derive in another way the Eulerian graphic generating function \eqref{e-A(x)} for acyclic digraphs by edges and descents, using  an interesting generalization of the chromatic polynomial.

Let us first sketch Stanley's approach. By applying the combinatorial interpretation of multiplication of graphic generating functions, we can show that for $\lambda$ a nonnegative integer, the coefficient of $x^n/n!\,(1+y)^{\binom n2}$ in 
\begin{equation}
\label{e-color}
\Biggl(\sumz n \dpowgy xn\Biggr)^\lambda
\end{equation}
counts ordered partitions of $[n]$ into $\lambda$ blocks (possibly empty), together with an arbitrary set of (undirected) edges between elements of different blocks, where each edge has weight $y$.  If we think of a vertex in the $i$th block as colored in color $i$, then we may describe these objects as $\lambda$-colored graphs, i.e., graphs in which the vertices are colored using colors chosen from $\{1,2,\dots, \lambda\}$ so that  adjacent vertices have different colors, as shown (for $y=1$) by Read \cite{read}. Thus for $\lambda$ a nonnegative integer, the coefficient of $x^n/n!\,(1+y)^{\binom n2}$ in \eqref{e-color} is the sum over all graphs $G$ on $[n]$ of $y^{\e(G)}\chi_G(\lambda)$, where $\e(G)$ is the number of edges of $G$. But if two polynomials in $\lambda$ are equal whenever $\lambda$ is a nonnegative integer then they are equal as polynomials in $\lambda$, and in particular, they are equal for $\lambda=-1$.
Applying Stanley's theorem on acyclic orientations, we find that setting $\lambda=-1$ in \eqref{e-color}, and replacing $x$ with $-x$,
gives the graphic generating function for acyclic digraphs. (Stanley only considered the case $y=1$ but the extension to counting by edges is straightforward.)

To include descents, we consider a generalization of the chromatic polynomial. Let $G$ be a graph with totally ordered vertices. 
We define a \emph{descent} of a proper coloring  $c$ of $G$ to be an edge $\{i,j\}$ of $G$  with $i<j$ and $c(i)>c(j)$.  We may define the \emph{refined chromatic polynomial} $X_G(\lambda)$ to be $\sum_{c}u^{\des(c)}$ where the sum is over all proper colorings $c$ of $G$ with colors chosen from  $\{1, 2, \dots, \lambda\}$, and $\des(c)$ is the number of descents of the coloring $c$. (It is not hard to show that $X_G(\lambda)$ is indeed a polynomial in $\lambda$; this follows from the proof of Theorem \ref{t-rcp} below.)
We note that $X_G(\lambda)$ is a specialization of the  chromatic quasisymmetric function introduced by Shareshian and Wachs \cite{sw}.

Then we have the following analogue of Stanley's theorem on acyclic orientations.

\begin{thm}
\label{t-rcp}
Let $G$ be a graph on a totally ordered $n$-element vertex set and let $X_G(\lambda)$ be the refined chromatic polynomial of $G$. Then 
\begin{equation*}
X_G(-1)=(-1)^n \sum_O u^{\des(O)}
\end{equation*}
where the sum is over all acyclic orientations $O$ of $G$.
\end{thm}

\begin{proof}[Proof sketch]
We follow closely Stanley's second proof in \cite{sacyclic} of the case $u=1$, to which we  refer for definitions not given here. To each proper coloring $c$ of $G$, we associate an acyclic orientation $O_c$ of $G$ in which each edge is directed from the lower-colored endpoint to the higher-colored endpoint. Then the coloring $c$ and the acyclic orientation $O_c$ have the same number of descents. As in Stanley's proof, the number of proper $\lambda$-colorings associated with a given acyclic orientation $O$ is the strict order polynomial $\bar\Omega(\bar O, \lambda)$, where $\bar O$ is the transitive and reflexive closure of $O$, regarded as a binary relation on the vertex set. 
Thus
\begin{equation*}
X_G(\lambda) = \sum_O u^{\des(O)}\bar\Omega(\bar O, \lambda),
\end{equation*}
where the sum is over all acyclic orientations of $G$. It is known, as a special case of Stanley's reciprocity theorem for order polynomials, that for every acyclic orientation $O$ we have
$\bar\Omega(\bar O, -1)=(-1)^n$, and the result follows.
\end{proof}

We can now give another proof of \eqref{e-A(x)}, counting acyclic digraphs by descents and edges.
Using Lemma \ref{l-qbinom}, we can show by induction on $\lambda$ that for $\lambda$ a nonnegative integer, the coefficient  of $x^n/n!_q (1+y)^{\binom n2}$ in 
\begin{equation*}
\Biggl(\sumz n \dpowbq xn\Biggr)^\lambda
\end{equation*}
counts $\lambda$-colored graphs on $[n]$, with edges weighted by $y$ and descents weighted by~$u$. Thus this coefficient is the sum 
\begin{equation*}
\sum_G y^{\e(G)}X_G(\lambda)
\end{equation*}
over all graphs $G$ on $[n]$. Setting $\lambda=-1$ and using
Theorem \ref{t-rcp} gives \eqref{e-A(x)}.

\subsection{Trees}\label{sec:trees}
We can use the basic idea of Lemma \ref{l-acyclic} to count rooted trees and forests by descents. 
Recall that we define a rooted tree (\emph{tree} for short) to be an acyclic digraph in which one vertex (the root) has outdegree 0 and every other vertex has outdegree 1. The vertices of indegree 0 are called \emph{leaves} but if the tree contains only one vertex, we do not consider this vertex to be a leaf. 
A (rooted) \emph{forest} is a digraph in which every weak component is a tree.
Let $t_n(u; \beta)$ be the sum of the weights of all trees with vertex set $[n]$, where the weight of a tree with $i$ descents and $j$ leaves is $u^i\beta^j$, and let $T(x,u; \beta) = \sumz n t_n(u; \beta) x^n/n!$.

We first illustrate the approach with $u=1$. The same approach to counting trees was taken in \cite{G96}. The result was stated there as a recurrence, but here we use exponential generating functions directly since  the analogue for general $u$, discussed  in Proposition \ref{p-tdl} below, would be more complicated as a recurrence.

Writing $T(x; \beta)$ for $T(x,1; \beta)$ (counting trees by leaves) and $T(x)$ for $T(x,1;1)$ (just counting trees), we will show that 
\begin{equation}
\label{e-tree1}
T(x; \beta+1) = T(xe^{\beta x}).
\end{equation}
The left side of \eqref{e-tree1} counts trees in which some subset of the leaves are marked, where each marked leaf is weighted $\beta$. To interpret the right side of \eqref{e-tree1}, we assume that the reader is familiar with the combinatorics of exponential generating functions, as described, for example, in \cite[Chapter 5]{ec2}.
The exponential generating function $xe^{\beta x}$ counts ``short trees": trees in which every vertex other than the root is a leaf, where the leaves are weighted by $\beta$. Then $T(xe^{\beta x}; 1)$
is the exponential generating function for structures obtained from rooted trees by replacing each vertex with a short tree. It is clear that these structures are essentially the same as the trees counted by the left side; the marked leaves corresponding to the leaves of the short trees.

Setting $\beta=-1$ in \eqref{e-tree1} gives $T(xe^{-x})=x$. In other words $T(x)$ is the compositional inverse of $xe^{-x}$, so $T(x)e^{-T(x)}=x$, or $T(x) = xe^{T(x)}$, the more common form of the functional equation for $T(x)$. These equations can be solved by Lagrange inversion or other methods to obtain the well-known formula 
\begin{equation*}
T(x) = \sum_{n=1}^\infty n^{n-1}\dpow xn,
\end{equation*}
and  more generally, 
\begin{equation*}
e^{zT(x)}=\sumz n z(z+n)^{n-1}\dpow xn,
\end{equation*}
which counts forests of rooted trees by the number of trees.
There is also a simple functional  equation for $T(x;\beta)=T(xe^{(\beta -1)x})$, which  counts trees by leaves. 
From the functional equation $T(x) = xe^{T(x)}$ we can easily obtain the functional equation for $T(x;\beta)=T(xe^{(\beta -1)x})$:
\begin{equation}
\label{e-T(x;alpha)}
T(x;\beta) = xe^{T(x;\beta) +(\beta-1)x}
\end{equation}
Equation \eqref{e-T(x;alpha)} is easy to see combinatorially, interpreting $T(x;\beta)+(\beta-1)x$ as counting trees by leaves, but now considering the root of a one-vertex tree to be a leaf.

Next, we can generalize \eqref{e-tree1} to keep track of descents. The argument is essentially the same as for \eqref{e-tree1} but we need to replace $xe^{\beta x}$ with something a little more complicated.
\begin{prop}
\label{p-tdl}
The exponential generating function $T(x,u;\beta)$ for trees by descents and leaves satisfies
\begin{equation}
\label{e-tree2}
T(x, u;\beta+1) = T\left(\frac{e^{\beta x} - e^{\beta u x}}{\beta(1-u)}, u\right),
\end{equation}
where  $T(x,u) = T(x,u;1)$ is the 
exponential generating function for trees by descents. Moreover, $T(x,u)$ is the
compositional inverse \textup{(}as a power series in $x$\textup{)} of  
\begin{equation*}
\frac{e^{-x} - e^{-u x}}{u-1} = \sum_{n=1}^\infty (-1)^{n-1}(1+u+\cdots +u^{n-1})\dpow xn.
\end{equation*}
\end{prop}
\begin{proof}
A short tree on $[n]$ with root $i+1$ has $i$ descents. Thus the exponential generating function for short trees, with descents weighted by $u$ and leaves weighted by $\beta$, is 
\begin{equation*}
\sum_{n=1}^\infty (1+u+\cdots +u^{n-1})\beta^{n-1}\dpow xn
 = \sum_{n=1}^\infty\frac{1-u^n}{1-u}\beta^{n-1}\dpow xn
 =\frac{e^{\beta x} - e^{\beta u x}}{\beta(1-u)}.
\end{equation*}
Then we obtain \eqref{e-tree2} in the same way that we obtained \eqref{e-tree1}.
As before, $T(x,u; 0) = x$, so setting $\beta=-1$ in \eqref{e-tree2} gives
\begin{equation*}
T\left(\frac{e^{-x} - e^{-u x}}{u-1}, u\right)=x.\qedhere
\end{equation*}
\end{proof}

Another combinatorial proof that $T(x,u)$ is the compositional inverse of $(e^{-x}-e^{-ux})/(u-1)$ was given by Drake \cite[Example 1.7.2]{drake}.

There is a simple formula for the coefficients of $T(x,u)$ that can be derived from our results and known formulas.

\begin{prop}
\label{p-trees}
For the exponential generating function $T(x,u)$ for trees by descents, we have the formulas
\begin{equation*}
T(x,u) = \sum_{n=1}^\infty \prod_{i=1}^{n-1} (iu+n-i)\dpow xn
\end{equation*}
and
\begin{equation*}
e^{zT(x,u)}=1+\sum_{n=1}^\infty z\prod_{i=1}^{n-1} (iu+n-i+z)\dpow xn.
\end{equation*}
\end{prop}

\begin{proof}
Since $T(x,u)$ is the compositional inverse of $(e^{-x} - e^{-u x})/(u-1)$, we have
\begin{equation*}
\frac{e^{-T(x,u)} - e^{-u T(x,u)}}{u-1} = x.
\end{equation*}
Multiplying both sides by $(1-u)e^{T(x,u)}$ gives
\begin{equation}
\label{e-Txu}
e^{(1-u)T(x,u)}-1 = (1-u)xe^{T(x,u)}.
\end{equation}
Now set $G=e^{(1-u)T(x,u)}$. Then \eqref{e-Txu} may be written
\begin{equation*}
G = 1+ (1-u)xG^{1/(1-u)},
\end{equation*}
and the desired formulas follow from the results of \cite[Section 5]{GS} or by Lagrange inversion (see, e.g., \cite[Section 3.3]{GLagrange}).
\end{proof}

We note that the formulas of Proposition \ref{p-trees} are proved by a different method in 
\cite[Section 9]{GS}, and more general enumerative results for trees have been proved bijectively by 
E\u gecio\u glu and Remmel \cite{ER}.

Forests have been counted by leaves and descents of a different kind in \cite{desc-leaves} but there does not seem to be any connection between the results described here and the results of \cite{desc-leaves}.

\textbf{Acknowledgment.} We would like to thank two anonymous referees for helpful comments.


\begin{thebibliography}{99}

\bibitem{pd}
\'E. de Panafieu and S. Dovgal,
Symbolic method and directed graph enumeration,
Acta Math. Univ. Comenian. (N.S.) 88 (2019), no. 3, 989–996. 


\bibitem{drake}
B. Drake, An Inversion Theorem for Labeled Trees and Some Limits of Areas Under Lattice Paths, Ph.D. thesis, Brandeis University, 2008.


\bibitem{ER}
\"O. E\u gecio\u glu and J. B.  Remmel, 
Bijections for Cayley trees, spanning trees, and their $q$-analogues,
J. Combin. Theory Ser. A 42 (1986), 15--30. 

\bibitem{qexp}
I. M. Gessel,
A $q$-analog of the exponential formula,
Discrete Math. 40 (1982), 69--80.

\bibitem{G96}
I. M. Gessel,
Counting acyclic digraphs by sources and sinks,
Discrete Math. 160 (1996), 253--258.


\bibitem{decomp}
I. M. Gessel,
Enumerative applications of a decomposition for graphs and digraphs, 
Discrete Math. 139 (1995), 257--271. 

\bibitem{desc-leaves}
I. Gessel,
Counting forests by descents and leaves,
The Foata Festschrift,
Electron. J. Combin. 3 (1996), no. 2, Research Paper 8, 5 pp. 

\bibitem{GLagrange}
I. M. Gessel,
Lagrange inversion,
J. Combin. Theory Ser. A 144 (2016), 212--249.

\bibitem{GS96}
I. M. Gessel and B. E. Sagan,
The Tutte polynomial of a graph, depth-first search, and simplicial complex partitions, 
The Foata Festschrift,
Electron. J. Combin. 3 (1996), no. 2, Research Paper 9, 36 pp. 


\bibitem{GS}
I. M. Gessel and S. Seo,
A refinement of Cayley's formula for trees,
Electron. J. Combin. 11 (2004/06), no. 2, Research Paper 27, 23 pp. 

\bibitem{L69}
V. A. Liskovec,
A recurrent method for the enumeration of graphs with labelled vertices (Russian),
Dokl. Akad. Nauk SSSR 184 (1969) 1284--1287. 
English translation in Soviet Math. Dokl. 10 (1969), 242--246.


\bibitem{mm}
J. W. Moon and L. Moser, 
Almost all tournaments are irreducible,
Canad. Math. Bull. 5 (1962), 61--65. 

\bibitem{ostroff}
J. Ostroff,
Counting Connected Digraphs with Gradings,
Ph.D. Thesis, Brandeis University, 2013.

\bibitem{read}
R. C. Read,
The number of $k$-colored graphs on labelled nodes,
Canad. J. Math. 12 (1960), 410--414. 

\bibitem{racyclic1}
R. W. Robinson, 
Enumeration of acyclic digraphs, 
in Proc. Second Chapel Hill Conf. on Combinatorial Mathematics and its Applications (Univ. North Carolina, Chapel Hill, N.C., 1970), pp. 391--399. Univ. North Carolina, Chapel Hill, N.C., 1970.

\bibitem{racyclic}
R. W. Robinson,
Counting labeled acyclic digraphs, 
in New Directions in the Theory of Graphs (Proc. Third Ann Arbor Conf., Univ. Michigan, Ann Arbor, Mich., 1971), pp. 239--273. Academic Press, New York, 1973. 

\bibitem{rstrong}
R. W. Robinson,
Counting digraphs with restrictions on the strong components, 
in Combinatorics and Graph theory '95, Vol. 1 (Hefei), 343--354, World Sci. Publ., River Edge, NJ, 1995. 

\bibitem{sw}
J. Shareshian and M. L.  Wachs,
Chromatic quasisymmetric functions,
Adv. Math. 295 (2016), 497--551. 

\bibitem{sacyclic}
R. P. Stanley,
Acyclic orientations of graphs, 
Discrete Mathematics 5 (1973),  171--178.

\bibitem{ec1}
R. P. Stanley, Enumerative Combinatorics, Vol. 1, Second edition, 
Cambridge Studies in Advanced Mathematics, vol. 49, Cambridge University Press, Cambridge, 2012.


\bibitem{ec2}
R. P. Stanley, Enumerative Combinatorics, Vol. 2, 
Cambridge Studies in Advanced Mathematics, vol. 62, Cambridge University Press, Cambridge, 1999.


\bibitem{wright}
E. M. Wright, 
The number of strong digraphs,
Bull. London Math. Soc. 3 (1971), 348--350. 

\end{thebibliography}
\end{document}